\documentclass{amsart}
\usepackage{geometry}
\usepackage{amssymb,latexsym,enumitem,bbm}
\usepackage[noadjust]{cite}
\usepackage{setspace}
\usepackage{color}
\usepackage{hyperref}
\usepackage{bbm}

\newcommand{\R}{\mathbb{R}}
\newcommand{\N}{\mathbb{N}}
\newcommand{\F}{\mathcal{F}}

\newcommand{\Sc}{\mathcal{S}}
\newcommand{\Exp}{\mathbb{E}}
\newcommand{\oo}{\mathcal{O}}
\newcommand{\K}{\mathcal{K}}

\newcommand{\inpr}[3][]{\left\langle#2 \,,\, #3\right\rangle_{#1}}

\newcommand{\indicator}[1]{\mathbbm{1}_{#1}}

\numberwithin{equation}{section}

\newtheorem{theorem}{Theorem}[section]

\newtheorem{lemma}[theorem]{Lemma}
\newtheorem{proposition}[theorem]{Proposition}
\newtheorem{remark}[theorem]{Remark}
\newtheorem{definition}[theorem]{Definition}

\newtheorem{example}[theorem]{Example}

\allowdisplaybreaks

\title[Stochastic PDEs]{Stochastic PDEs in $\mathcal{S}^\prime$ for SDEs driven by L\'evy noise}

\author{Suprio Bhar}
\address{Suprio Bhar, Tata Institute of Fundamental Research, Centre For Applicable Mathematics, Post Bag No 6503, GKVK Post Office, Sharada Nagar, Chikkabommsandra, Bangalore 560065, India.}
\email{suprio@tifrbng.res.in, suprio.bhar@gmail.com}

\author{Rajeev Bhaskaran}
\address{Rajeev Bhaskaran, Indian Statistical Institute Bangalore Centre, 8th Mile Mysore Road, Bangalore 560059, India.}
\email{brajeev@isibang.ac.in}

\author{Barun Sarkar}
\address{Barun Sarkar, Indian Statistical Institute Bangalore Centre, 8th Mile Mysore Road, Bangalore 560059, India.}
\email{barunsarkar.math@gmail.com}

\begin{document}

\begin{abstract}
In this article we show that a finite dimensional stochastic differential equation driven by a
L\'evy process can be formulated as a stochastic partial differential
equation. We prove the existence and uniqueness of strong solutions of such stochastic PDEs. The solutions that we construct have the `translation invariance' property. The special case of this
correspondence for diffusion processes was proved in [Rajeev, \emph{Translation invariant diffusion in the space of
tempered distributions}, Indian J. Pure Appl. Math. \textbf{44} (2013), no.~2,
  231--258].
\end{abstract}

\maketitle

\section{Introduction}\label{intro}
In this article we show that a finite dimensional stochastic differential equation (abbrev. SDE) driven by a
L\'evy process can be formulated as a stochastic partial differential
equation (abbrev. SPDE). The goal of this article is to prove the existence and uniqueness of strong solutions of such SPDEs.

Given a Brownian motion $\{B_t\}$ and an independent Poisson random measure $N$ driven by a L\'evy measure $\nu$, we consider an SDE in $\R^d$, of the form
\begin{equation}\label{sde}
\begin{split}
U_{t} &= \kappa + \int_0^t\bar b(U_{s-};\xi)ds + \int_0^t \bar\sigma(U_{s-};\xi)\cdot dB_s + \int_0^t\int_{(0 < |x| < 1)}  \bar F(U_{s-},x;\xi)\, \widetilde
N(dsdx)\\
&+ \int_0^t \int_{(|x| \geq  1)} \bar G(U_{s-},x;\xi) \,
N(dsdx), \quad t\geq 0,
\end{split}
\end{equation}
and the corresponding SPDE in the space of tempered distributions $\Sc^\prime$, (more specifically, in a Hermite-Sobolev space $\Sc_{-p}$) viz.
\begin{equation}\label{spde}
\begin{split}
Y_t
&= \xi + \int_0^t
A(Y_{s-})\cdot dB_s + \int_0^t \widetilde L(Y_{s-})\, ds\\
&+ \int_0^t \int_{(0 < |x| < 1)} \left(\tau_{F(Y_{s-},x)}
-Id\right) \,
Y_{s-}\,\widetilde N(dsdx) + \int_0^t \int_{(|x| \geq  1)} \left(\tau_{G(Y_{s-},x)}
-Id\right) \,Y_{s-}\,
N(dsdx),
\end{split}
\end{equation}
where $\xi$ is an $\Sc_{-p}$ valued $\F_0$-measurable random variable and $\kappa$ is an $\R^d$ valued $\F_0$-measurable random variable (see Sections \ref{sec:1}, \ref{sec:2} and \ref{sec:3} for notations).

The study of such correspondence was initiated in
\cite{MR1837298},\cite{MR3063763} for diffusion processes with a
deterministic initial condition. In \cite{MR3687773}, this was
extended to random initial conditions, which require some technical
conditions on the coefficients of the diffusion processes. For
diffusion processes this correspondence together with the pathwise
uniqueness of \eqref{spde} actually leads to strong solutions of \eqref{sde} when
the diffusion and drift coefficients are `rough' (see \cite[Proposition 4.1]{MR3517629}). The question of when the solutions of an SPDE
can be realised on finite dimensional submanifolds has also been

studied recently in the context of the HJM model in finance (see
\cite{MR1849252, MR1822777, MR3227060}). A key feature of our
correspondence and of independent interest is that the diffusion,
drift and jump coefficients viz. $\bar\sigma, \bar b, \bar F, \bar
G$ for the finite dimensional SDE can be written as a convolution
involving the initial condition $\xi$ of the SPDE (see \cite[Remark
3.7]{MR3063763}). The correspondence between SPDE's and SDE's also
extends to the flows generated by the SDE's (see \cite{MR2373102,
mild-soln}).

In this paper we show the local existence and uniqueness of strong
solutions to the SPDE (Theorems \ref{uniq-via-monotonicity} and \ref{complt-eqn}).
The existence is shown by an explicit construction of the solution
$\{Y_t\}$ of the SPDE
as a translate of the initial condition $\xi$ by the solutions $\{U_t\}$ of the finite dimensional SDE i.e. $Y_t = \tau_{U_t}(\xi)$. Here $\tau_x : \mathbb R^d \rightarrow \mathbb R^d, x \in \R^d$ denote the translation operators. This requires a
generalisation of the It\^o formula in \cite{MR1837298} for continuous semi-martingales with a non-random initial condition to semi-martingales
with jumps and in particular to L\'evy processes. This was done in \cite{MR3647067}, where an existence theorem for the SPDE \eqref{spde} was also proved for a sub class of SDE's than those considered here (\cite[Theorem 4.7]{MR3647067},
\cite{JOTP-erratum}). See also \cite{MR664333} for a related It\^o formula.

The uniqueness result uses the technique of the `Monotonicity inequality' (see \cite{MR3331916, MR2590157, MR570795}). The main difference
between the present case and the
cases treated earlier in the references above is the addition of
the jump terms in  the SPDE. The large jumps are easily handled by a boundedness assumption. Estimate for the small jump terms (see Term 2 in \eqref{Exp-norm-Ito}) require the 'Monotonicity inequality' and involves a second order Taylor expansion of the functions of the form
$v \in [0,1] \rightarrow <\tau_{vz}\psi, \phi>$ where $\psi$ is a tempered distribution, $\phi$ is a suitable test function and $z \in \mathbb R^d$, thereby reducing it to the case of second order constant coefficient differential operators as in
\cite{MR3331916}. If the L\'evy measure $\nu$ is bounded, then the required estimate follows provided the coefficient appearing in the small jump terms is bounded (see Remark \ref{term2-via-mon}).

The proofs for the existence and uniqueness of local strong solutions of the SDEs use standard techniques (e.g. those used in \cite{MR2512800, MR2560625}), but requires growth and continuity assumptions of the coefficients which involve an additional parameter. The assumptions are stated in \ref{sigma-b}, \ref{loc-Lip}, \ref{F1}, \ref{F2},\ref{F3} and \ref{G1} and the proofs using these assumptions are given in \cite{fd-SDE}.

Our proof of local existence and
uniqueness for the SPDE involves conditions \ref{F1}, \ref{F2}, \ref{F3} and \ref{G1}-\ref{G2} on the coefficients involving the small and large jumps (Theorem \ref{uniq-via-monotonicity}). In Theorem \ref{complt-eqn}, using the well known `interlacing technique' for L\'evy processes, we eliminate the condition \ref{G2} involving the large jumps.

In this article we restrict ourselves to the study of the
relationship between \eqref{sde} and \eqref{spde}. Our results are
proved in the framework of the Hermite-Sobolev spaces. For
formulations of SPDEs in these spaces see \cite{MR1465436,
MR771478}. We also refer to \cite{MR3235846} for an SPDE associated
to branching measure valued processes and canonically linked to
Brownian motion, and to \cite{MR1242575, MR1280712} for the
connection between measure valued processes and finite dimensional
diffusions. In \cite{MR3075420}, an analogous class of processes
arise in the study of interacting particle systems wherein the
finite dimensional SDE represents the microscopic motion of a
`tagged' particle and the SPDE describes the macroscopic behaviour
of a system of particles.

\section{Preliminaries}\label{sec:1}
\subsection{Topology}\label{sec:1-1}
Let $\Sc$ be the space of rapidly decreasing smooth functions on $\R^d$ with dual $\Sc^\prime$, the space of tempered distributions (see \cite{MR771478}). Let $\mathbb{Z}^d_+:=\{n=(n_1,\cdots, n_d): \; n_i \text{ non-negative integers}\}$. If $n\in\mathbb{Z}^d_+$, we define $|n|:=n_1+\cdots+n_d$.

For $p \in \R$, consider the increasing norms $\|\cdot\|_p$, defined by the inner
products
\begin{equation}
\langle f,g\rangle_p:=\sum_{n\in\mathbb{Z}^d_+}(2|n|+d)^{2p}\langle f,h_n\rangle\langle g,h_n\rangle,\ \ \ f,g\in\Sc.
\end{equation}
In the above equation, $\{h_n: n\in\mathbb{Z}^d_+\}$ is an orthonormal basis for $\mathcal{L}^2(\R^d,dx)$ given by the Hermite functions and $\langle\cdot,\cdot\rangle$ is the usual
inner product in $\mathcal{L}^2(\R^d,dx)$. For $d=1$,
$h_n(t) :=(2^n n!\sqrt{\pi})^{-1/2}\exp\{-t^2/2\}H_n(t)$, where $H_n, t \in \R$ are the Hermite polynomials (see \cite{MR771478}). For $d > 1$, $h_n(x_1,\cdots,x_d) := h_{n_1}(x_1)\cdots h_{n_d}(x_d)$ for all $(x_1,\cdots,x_d) \in \R^d, n\in\mathbb{Z}^d_+$, where the Hermite functions on the right hand side are one-dimensional. We define the Hermite-Sobolev spaces $\Sc_p, p \in \R$ as the completion of $\Sc$ in
$\|\cdot\|_p$. Note that the dual space $\Sc_p^\prime$ is isometrically isomorphic with $\Sc_{-p}$ for $p\geq 0$ and $\inpr{\cdot}{\cdot}$ extends the $\mathcal{L}^2$ inner product to the duality between $\Sc$ and $\Sc^\prime$. We also have $\Sc = \bigcap_{p}(\Sc_p,\|\cdot\|_p), \Sc^\prime=\bigcup_{p>0}(\Sc_{-p},\|\cdot\|_{-p})$ and $\Sc_0 = \mathcal{L}^2(\R^d)$.

Consider the derivative maps denoted by $\partial_i:\Sc\to
\Sc$ for $i=1,\cdots,d$. We can extend these maps by duality to
$\partial_i:\Sc' \to \Sc'$ as follows: for $\psi \in
\Sc'$,
\[\inpr{\partial_i \psi}{\phi}:=-\inpr{\psi}{\partial_i \phi}, \; \forall \phi
\in \Sc.\]
Let $\{e_i: i=1,\cdots,d\}$ be the standard basis vectors in $\R^d$. Then
for any $n =
(n_1,\cdots,n_d) \in {\mathbb Z}_+^d$ we have (see
\cite[Appendix A.5]{MR562914})
\[\partial_i h_n =
\sqrt{\frac{n_i}{2}}h_{n-e_i}-\sqrt{\frac{n_i+1}{2}}h_{n+e_i},\]
with the convention that for a multi-index $n = (n_1,\cdots,n_d)$, if $n_i
< 0$ for some $i$, then $h_n \equiv 0$. The above recurrence relation implies that
$\partial_i:\Sc_{p}\to\Sc_{p-\frac{1}{2}}$ is a bounded linear
operator.\\
For $x \in \R^d$, let $\tau_x$ denote the translation operators on $\Sc$
defined by
$(\tau_x\phi)(y):=\phi(y-x), \, \forall y \in \R^d$. These operators can be
extended to $\tau_x:\Sc'\to \Sc'$ by
\[\inpr{\tau_x\phi}{\psi}:=\inpr{\phi}{\tau_{-x}\psi},\, \forall \psi \in
\Sc.\]
For $x \in \R^d$, $|x|$ will denote its Euclidean norm.
\begin{proposition}\label{tau-x-estmte}
The translation operators $\tau_x, x \in \R^d$ have the following properties:
\begin{enumerate}[label=(\alph*)]
\item For $x \in \R^d$ and any $p \in \R$, $\tau_x: \Sc_p\to\Sc_p$
is a bounded linear map. In particular, there exists a real polynomial $P_k$ of
degree $k = 2(\lfloor|p|\rfloor +1)$ such that
\[\|\tau_x\phi\|_p\leq P_k(|x|)\|\phi\|_p, \, \forall \phi \in \Sc_p.\]
\item For any $x \in \R^d$ and any $i=1,\cdots,d$ we have $\tau_x\partial_i = \partial_i\tau_x$.
\item Fix $\phi \in \Sc_p$ for some $p \in \R$. The map $x \in\R^d \mapsto \tau_x\phi \in \Sc_p$ is continuous.
\end{enumerate}
\end{proposition}
\begin{proof}
See \cite[Theorem 2.1]{MR1999259} for the proof of part $(a)$ and the proof of \cite[Proposition 3.1]{MR2373102} for the proof of part $(c)$. Part
$(b)$ is well-known.\end{proof}

\begin{proposition}[{\cite[Proposition 3.8]{MR3687773}}]\label{ext-tau}
Let $p>d+\frac{1}{2}$. Then for any $\psi \in \Sc_{p+\frac{1}{2}}$ and any positive integer $n$, there exists a constant $D(n)>0$ such that for all $x_1,x_2\in \R^d$ with $|x_1|, |x_2| \leq n$, we have
\begin{equation}
\|\tau_{x_1}\psi-\tau_{x_2}\psi\|_p\leq D(n)\|\psi\|_{p+\frac{1}{2}} |x_1-x_2|.
\end{equation}
In particular, for any bounded set $\K$ in $\Sc_{p+\frac{1}{2}}$ and any positive integer $n$, there exists a constant $D(\K,n)>0$ such that for all $x_1,x_2\in \R^d$ with $|x_1|, |x_2| \leq n$, we have
\begin{equation}\label{tau1}
\|\tau_{x_1}\psi-\tau_{x_2}\psi\|_p\leq D(\K,n) |x_1-x_2|, \ \forall \psi\in \K.
\end{equation}
\end{proposition}
\begin{proof}
The proof is contained in the proof of \cite[Proposition 3.8]{MR3687773}, specifically, in the arguments after \cite[equation (3.16)]{MR3687773}.
\end{proof}

Let $\sigma =
(\sigma_{ij})$ be a constant $d \times r$ matrix with $(a_{ij})=(\sigma
\sigma^t)_{ij}$ and $b = (b_1,...,b_d) \in \R^d$. For $\phi \in \Sc$, we define
\[
\begin{array}{l}
L\phi:=\frac{1}{2}\sum_{i,j=1}^d a_{ij}\partial_{ij}^2 \phi - \sum_{i=1}^d
b_i\partial_i \phi,\\
A_i\phi:=-\sum_{j=1}^d \sigma_{ji}(\partial_j \phi),\, i = 1,\cdots,r\\
A\phi=(A_1\phi,\dots,A_r\phi)
\end{array}
\]
\begin{theorem}[{\cite[Theorem 2.1 and Remark 3.1]{MR2590157}}]\label{constant-monotonicity} For every $p \in \R$, there exists a positive constant $C = C(p,d,(\sigma_{ij}),(b_j))$,
such that
\begin{equation}
2\inpr[p]{\phi}{L\phi} + \|A\phi\|_{HS(p)}^2 \leq
C.\|\phi\|_p^2
\end{equation}
for all $\phi \in \Sc$, where $\|A\phi\|_{HS(p)}^2:= \sum_{i=1}^r
\|A_i\phi\|_p^2$. Furthermore, by density arguments the above inequality can
be extended to all $\phi \in \Sc_{p+1}$. The constant $C$ depends on $\sigma_{ij}, b_i$ through the maximum of $|\sigma_{ij}|, |b_i|$ and hence the inequality can be extended to the case where $\sigma, b$ are bounded processes parametrized by some set.
\end{theorem}

\subsection{An It\^{o} formula}\label{sec:1-2}
Let $\big(\Omega,\F,\{\F_t\}_{t\geq0},P\big)$ be a filtered complete probability space satisfying the usual conditions viz. $\F_0$ contains all $A\in\F$, s.t. $P(A)=0$ and $\F_t=\bigcap_{s>t}\F_s, t \geq 0$. Given two real valued semimartingales $\{X^1_t\}$ and $\{X^2_t\}$, let $\{[X^1,X^2]_t^c\}$ denote the continuous part of the covariation process $\{[X^1,X^2]_t\}$.

\begin{theorem}\label{random-initial}
Let $p>0$. Let $\xi$ be an $\Sc_{-p}$ valued $\F_0$ measurable random variable. Let $\{X_t\}$ be an $\R^d$
valued $(\F_t)$ semimartingale with $X_t=(X^1_t,\cdots,X^d_t)$. Then
$\{\tau_{X_t}\xi\}$ is an $\Sc_{-p}$ valued semimartingale and
\[\sum_{s \leq t}\left[\tau_{X_s}\xi - \tau_{X_{s-}}\xi +
\sum_{i=1}^d (\bigtriangleup X^i_s\,\partial_i\tau_{X_{s-}}\xi)\right]\] is an
$\Sc_{-p-1}$ valued process of finite variation and we have the following
equality
in
$\Sc_{-p-1}$, a.s.
\begin{equation}\label{Ito-formula-random-initl}
\begin{split}
\tau_{X_t}\xi &= \tau_{X_0}\xi - \sum_{i=1}^d \int_0^t
\partial_i\tau_{X_{s-}}\xi\,
dX^i_s + \frac{1}{2}\sum_{i,j=1}^d \int_0^t \partial_{ij}^2\tau_{X_{s-}}\xi\,
d[X^i,X^j]^c_s\\
&+\sum_{s \leq t}\left[\tau_{X_s}\xi - \tau_{X_{s-}}\xi +
\sum_{i=1}^d (\bigtriangleup X^i_s\,\partial_i\tau_{X_{s-}}\xi)\right], \, t
\geq 0.
\end{split}
\end{equation}
\end{theorem}
\begin{proof}
The case when $\xi$ is deterministic was proved in \cite[Theorem 4.5]{MR3647067}. We indicate the proof for a random $\xi$ via two observations.
\begin{enumerate}[label=(\roman*)]
\item Recall that $\partial_i:\Sc_{q}\to\Sc_{q-\frac{1}{2}}, 1 \leq i \leq d$ are bounded linear
operators for every $q \in \R$.
By Proposition \ref{tau-x-estmte}, the processes $\{\tau_{X_{t-}}\xi\}, \{\partial_i\tau_{X_{t-}}\xi\}, \{\partial^2_{ij}\tau_{X_{t-}}\xi\}, 1 \leq i,j \leq d$ are locally norm bounded predictable processes with values in $\Sc_{-p}, \Sc_{-p-\tfrac{1}{2}}$ and $\Sc_{-p-1}$ respectively. Hence the integrals $\{\int_0^t
\partial_i\tau_{X_{s-}}\xi\,
dX^i_s\}, \{\int_0^t \partial_{ij}^2\tau_{X_{s-}}\xi\,
d[X^i,X^j]^c_s\}, 1 \leq i,j \leq d$ exist.
\item Given any $\F_0$ measurable set $F$, an $\Sc_{-p}$ valued predictable step process $\{G_t\}$ and an $\R^d$ valued rcll semimartingale $\{X_t\}$, we have a.s.
\begin{equation}\label{indicator-F}
\mathbbm{1}_F\int_0^tG_sdX_s=\int_0^t\mathbbm{1}_FG_sdX_s,\ \ \ t\geq0.
\end{equation}
This equality can be extended to the case involving locally norm-bounded $\Sc_{-p}$ valued predictable process $\{G_t\}$. Again, given any $\F_0$ measurable set $F$, $\phi\in\Sc_{-p}$, $\psi\in\Sc$ and $x\in\R^d$ we have
\begin{equation}\label{indicator-dual}
\begin{split}
\langle\mathbbm{1}_F\tau_x\phi,\psi\rangle&=\mathbbm{1}_F\langle\tau_x\phi,\psi\rangle=\mathbbm{1}_F\langle\phi,\tau_{-x}\psi\rangle\\
&=\langle\mathbbm{1}_F\phi,\tau_{-x}\psi\rangle=\langle\tau_x(\mathbbm{1}_F\phi),\psi\rangle
\end{split}
\end{equation}
and hence $\mathbbm{1}_F\tau_x\phi=\tau_x(\mathbbm{1}_F\phi)$. Similarly $\mathbbm{1}_F\tau_x\phi=\tau_{\mathbbm{1}_F(x)}(\mathbbm{1}_F\phi)$.
\end{enumerate}

Since $\Omega = \bigcup_{M=1}^\infty \{\omega: \|\xi(\omega)\|_{-p} \leq M\}$, it is enough to establish \eqref{Ito-formula-random-initl} for almost every $\omega$ in $\{\omega: \|\xi(\omega)\|_{-p} \leq M\}$ for every fixed positive integer $M$. Multiplying \eqref{Ito-formula-random-initl} by $\mathbbm{1}_{\{\|\xi\|_{-p} \leq M\}}$, it is enough to establish the result when $\xi$ is norm bounded.

From \cite[Theorem 4.5]{MR3647067} and \eqref{indicator-F},  \eqref{indicator-dual} we can establish the required result when $\xi$ is an $\Sc_{-p}$ valued simple $\F_0$ measurable random variable. A limiting argument then proves the result when $\xi$ is norm bounded. This completes the proof.
\end{proof}

\section{Finite dimensional SDEs}\label{sec:2}
\subsection{setup and notations}\label{sec:2-1}
We use the following notations throughout the paper.
\begin{itemize}
\item The set of positive integers will be denoted by $\N$. Recall that for $x \in \R^n$, $|x|$ denotes its Euclidean norm. The transpose of any element $x \in \R^{n\times m}$ will be denoted by $x^t$.
\item For any $r > 0$, define $\oo(0,r):=\{x \in \R^d: |x|< r\}$. Then $\overline{\oo(0,r)} = \{x \in \R^d: |x| \leq r\}$ and $\oo(0,r)^c = \{x \in \R^d: |x|\geq r\}$.
\item Let $\big(\Omega,\F,\{\F_t\}_{t\geq0},P\big)$ be a filtered complete probability space satisfying the usual conditions viz. $\F_0$ contains all $A\in\F$, s.t. $P(A)=0$ and $\F_t=\bigcap_{s>t}\F_s, t \geq 0$.
\item Let $p>0$. Let $\sigma = (\sigma_{ij})_{d\times d}, b=(b_1, \cdots, b_d)^t$ be such that $\sigma_{ij}, b_i:\Omega\to\Sc_{p}$ are $\F_0$ measurable and
\[\beta:= \sup\{\|\sigma_{ij}(\omega)\|_p, \|b_i(\omega)\|_p:\omega \in \Omega, 1 \leq i,j \leq d\} < \infty.\tag*{$\mathbf{(\sigma b)}$} \label{sigma-b}\]
\item Define $\bar\sigma:\Omega\times\R^d\times\Sc_{-p}\to \R^{d\times d}$ and $\bar b:\Omega\times\R^d\times\Sc_{-p}\to \R^d$ by $\bar\sigma(\omega,z;y) := \inpr{\sigma(\omega)}{\tau_z y}$ and $\bar b(\omega,z;y) :=
\inpr{b(\omega)}{\tau_z y}$, where $(\inpr{\sigma(\omega)}{\tau_z y})_{ij}:= \inpr{\sigma_{ij}(\omega)}{\tau_z y}$ and $(\inpr{b(\omega)}{\tau_z y})_i := \inpr{b_i(\omega)}{\tau_z y}$.
\item Let $F:\Omega\times\mathcal{S}_{-p}\times \oo(0,1) \to
\R^d$ and $G:\Omega\times\mathcal{S}_{-p}\times \oo(0,1)^c \to
\R^d$ be $\F_0\otimes \mathcal{B}(\Sc_p)\otimes\mathcal{B}(\oo(0,1))/\mathcal{B}(\R^d)$ and $\F_0\otimes \mathcal{B}(\Sc_p)\otimes\mathcal{B}(\oo(0,1)^c)/\mathcal{B}(\R^d)$ measurable respectively. Here $\mathcal{B}(\K)$ denotes the Borel $\sigma$-field of set $\K$.
\item Define $\bar F: \Omega\times\R^d\times \oo(0,1)\times\Sc_{-p}\to \R^d$, $\bar G: \Omega\times\R^d \times \oo(0,1)^c\times\Sc_{-p}\to \R^d$ by $\bar F(\omega,z,x;y) := F(\omega,\tau_z y,x), \ \bar G(\omega,z,x;y) := G(\omega,\tau_z y, x)$.
\item Let $\{B_t\}$ denote a standard Brownian motion and let $N$ denote a Poisson random measure driven by a L\'evy measure $\nu$.  $\widetilde N$ will denote the corresponding compensated random measure. We also assume that $B$ and $N$ are independent.
\item In our arguments, at times we use time intervals of the form $[0,T]$. In such cases, $T$ will always assumed to be finite i.e. $[0,T]$ will be a finite time interval.
\item Given a process $\{X_t\}$ and a stopping time $\eta$, the stopped process $\{X^\eta_t\}$ is defined as $X^\eta_t := X_{t \wedge \eta}$.
\end{itemize}

Consider the following SDE in $\R^d$,
\begin{equation}\label{fd-sde-sln}
\begin{split}
dU_{t} &=\bar b(U_{t-};\xi)dt+ \bar\sigma(U_{t-};\xi)\cdot dB_t +\int_{(0 < |x| < 1)}  \bar F(U_{t-},x;\xi)\, \widetilde
N(dtdx)\\
&+\int_{(|x| \geq  1)} \bar G(U_{t-},x;\xi) \,
N(dtdx), \quad t\geq 0 \\
U_0&=\kappa,
\end{split}
\end{equation}
where $\xi$ is an $\Sc_{-p}$ valued $\F_0$-measurable random variable and $\kappa$ is an $\R^d$ valued $\F_0$-measurable random variable. Unless stated otherwise, $\xi$ and $\kappa$ will be taken to be independent of the noise $B$ and $N$. Note that the $i$-th component of $\int_0^t \bar\sigma(U_{s-};\xi)\cdot dB_s$ is $\sum_{j=1}^d\int_0^t \bar\sigma_{ij}(U_{s-};\xi)\, dB^j_s$. We list some hypotheses.

\begin{enumerate}[label=\textbf{(F\arabic*)},ref=\textbf{(F\arabic*)}]
\item\label{F1} For all $\omega\in\Omega$ and $x \in \oo(0,1)$ there exists a constant $C_x \geq 0$ s.t.
\begin{equation}\label{asm1}
\lvert F(\omega,y_1,x)-F(\omega,y_2,x)\rvert\leq C_x\|y_1-y_2\|_{-p-\frac{1}{2}}, \forall y_1,y_2\in\Sc_{-p}.
\end{equation}
We assume $C_x$ to depend only on $x$ and independent of $\omega$. Since $\|y\|_{-p-\frac{1}{2}} \leq \|y\|_{-p}, \forall y \in \Sc_{-p}$, we have
\[\lvert F(\omega,y_1,x)-F(\omega,y_2,x)\rvert\leq C_x\|y_1-y_2\|_{-p}, \forall y_1,y_2\in\Sc_{-p}.\]
\item\label{F2} The constant $C_x$ mentioned above has the following properties, viz.
\[\sup_{|x|<1}C_x<\infty,\quad \int_{(0<|x|<1)}C_x^2\, \nu(dx)<\infty.\]
\item\label{F3} $\sup_{\omega\in\Omega,|x|<1}|F(\omega,0,x)|<\infty$ and  $\sup_{\omega\in\Omega}\int_{(0<|x|<1)} |F(\omega,0,x)|^2\,\nu(dx)<\infty$. \end{enumerate}

\begin{enumerate}[label=\textbf{(G\arabic*)},ref=\textbf{(G\arabic*)}]
\item\label{G1} The mapping $y\rightarrow G(\omega,y,x)$ is continuous for all $x \in \oo(0,1)^c$ and $\omega \in \Omega$.
\item\label{G2} For every bounded set $\K$ in $\Sc_{-p}$,
\[\sup_{\substack{\omega\in\Omega, y \in \K,\\ x \in O(0,1)^c}} |G(\omega,y,x)| < \infty.
\]
\end{enumerate}

\begin{example}
Examples of coefficient $F$ satisfying \ref{F1}, \ref{F2} and \ref{F3} can be constructed as follows. Choose a function $h:\Omega \to \R$ which is bounded and $\F_0$ measurable. Next choose Borel measurable $f_1:\oo(0,1)\to\R$ with $f_1\in {\mathcal L^2}(\oo(0,1), \nu)\cap {\mathcal L^\infty}$. Fix $\gamma_1,\cdots,\gamma_d \in \Sc_{p + \frac{1}{2}}$ and consider the function $f_2: \Sc_{-p} \to \R^d$ defined by $f_2(y) := (\inpr{\gamma_1}{y},\cdots,\inpr{\gamma_d}{y})^t$. Note that $\inpr{\gamma_1}{y}$ etc. are duality actions, since $\Sc_{-p} \subset \Sc_{- p - \frac{1}{2}} \cong (\Sc_{p + \frac{1}{2}})^\prime$ and hence $f_2$ is Lipschitz in the $\|\cdot\|_{-p-\frac{1}{2}}$ norm. Then the function $F(\omega,y,x) := h(\omega) f_1(x) f_2(y)$ satisfies the required assumptions. Examples of coefficient $G$ satisfying \ref{G1} and \ref{G2} can be constructed as follows. Take any bounded Borel measurable function $g_1:\oo(0,1)^c\to\R$ and let $h$ be as above. Fix $\gamma_1,\cdots,\gamma_d \in \Sc_p$ and consider the function $g_2: \Sc_{-p}\to\R^d$ defined by $g_2(y):= (\inpr{\gamma_1}{y},\cdots,\inpr{\gamma_d}{y})^t$. Then the function $G(\omega,y,x) := h(\omega) g_1(x) g_2(y)$ satisfies the required assumptions. Finite linear combinations of such functions are also examples of $F$ and $G$.
\end{example}

We also require certain Lipschitz regularity of the coefficients of \eqref{fd-sde-sln}. For the sake of convenience, we state the hypothesis here.

(Locally Lipschitz in $z$, locally in $y$) For every bounded set $\K$ in $\Sc_{-p}$ and positive integer $n$ there exists a constant $C(\K,n)>0$ such that for all $z_1, z_2\in \oo(0,n),\ y\in \K$ and $\omega\in\Omega$
\begin{equation}\label{loc-Lip}\tag*{\textbf{(loc-Lip)}}
\begin{split}
&|\bar{b}(\omega,z_1;y) - \bar{b}(\omega,z_2;y)|^2+ |\bar{\sigma}(\omega,z_1;y)-
\bar{\sigma}(\omega,z_2;y)|^2\\
&+\int_{(0 < |x| < 1)}|\bar{F}(\omega,z_1,x;y) - \bar{F}(\omega,z_2,x;y)|^2 \, \nu(dx) \leq C(\K,n)\,
|z_1
- z_2|^2.
\end{split}
\end{equation}

We now prove boundedness properties of the coefficient $F$ which follow from our hypotheses.

\begin{lemma}\label{f-bd}
Let \ref{F1}, \ref{F2} and \ref{F3} hold. Then, for any bounded set $\K$ in $\Sc_{-p}$ the following are true.
\begin{enumerate}[label=(\roman*)]
\item $\sup_{\omega\in\Omega,y\in \K,|x|<1}|F(\omega,y,x)|<\infty$.
\item $\sup_{\omega\in\Omega,y\in \K}\int_{(0<|x|<1)}|F(\omega,y,x)|^2\nu(dx) =: \alpha(\K) <\infty$.
\item $\sup_{\omega\in\Omega,y\in \K}\int_0^t\int_{(0<|x|<1)}|F(\omega,y,x)|^4\nu(dx)ds<\infty$ for all $0\leq t<\infty$.
\end{enumerate}
\end{lemma}

\begin{proof}
For $y\in\Sc_{-p}$,
\begin{equation}\label{asm2}
|F(\omega,y,x)| \leq|F(\omega,y,x)-F(\omega,0,x)|+|F(\omega,0,x)| \leq C_x\|y\|_{-p}+|F(\omega,0,x)|.
\end{equation}
Then, for any bounded set $\K$ in $\Sc_{-p}$,
\begin{equation}\label{con1}
\sup_{\omega\in\Omega,y\in \K,|x|<1}|F(\omega,y,x)|\leq (\sup_{|x|<1}C_x)\cdot(\sup_{y\in \K}\|y\|_{-p})+\sup_{\omega\in\Omega,|x|<1}|F(\omega,0,x)| <\infty.
\end{equation}
Now,
\begin{equation}\label{asm3}
|F(\omega,y,x)|^2\leq 2C_x^2\|y\|^2_{-p}+2|F(\omega,0,x)|^2.
\end{equation}
Then for $y\in\K$,
\begin{equation}\label{con2}
\begin{split}
&\int_{(0<|x|<1)}|F(\omega,y,x)|^2\nu(dx)\\
&\leq 2\sup_{y\in \K}\|y\|^2_{-p}\int_{(0<|x|<1)}C_x^2\nu(dx)+2\int_{(0<|x|<1)}|F(\omega,0,x)|^2\nu(dx).
\end{split}
\end{equation}
From \eqref{con2}, (ii) follows.
Combining part (i) and (ii), (iii) follows. This completes the proof.
\end{proof}

Using the continuity result in Proposition \ref{tau-x-estmte} the next result follows.

\begin{lemma}
Suppose \ref{G1} holds. Then the map $z \in \R^d \to \bar G(\omega, z,x; \xi(\omega)) = G(\omega,\tau_z\xi(\omega),x) \in \R^d$ is continuous for all $x \in \oo(0,1)^c$ and $\omega \in \Omega$.
\end{lemma}

\subsection{Global Lipschitz coefficients}\label{sec:2-2}
We first consider the existence and uniqueness of solutions for the reduced equation, viz.
\begin{equation}\label{reduced-fd-sde}
\begin{split}
dU_{t} &=\bar b(U_{t-};\xi)dt+ \bar\sigma(U_{t-};\xi)\cdot dB_t +\int_{(0 < |x| < 1)}  \bar F(U_{t-},x;\xi)\, \widetilde
N(dtdx), \quad t\geq 0 \\
U_0&=\kappa;
\end{split}
\end{equation}
with $\xi$ and $\kappa$ as in \eqref{fd-sde-sln}.

\begin{theorem}[{\cite{fd-SDE}}]
\label{nrm-bd-rndm-inl}
Let \ref{sigma-b}, \ref{F1}, \ref{F2} and \ref{F3} hold. Suppose the following conditions are satisfied.
\begin{enumerate}[label=(\roman*)]
\item $\kappa, \xi$ are $\F_0$ measurable, as stated in \eqref{fd-sde-sln}.
\item (Global Lipschitz in $z$, locally in $y$) For every bounded set $\K$ in $\Sc_{-p}$, there exists a constant $C(\K)>0$ such that for all $z_1, z_2\in\R^d,\ y\in \K$ and $\omega\in\Omega$
\begin{equation}\label{Lipschitz-condition-rnm-inl}
\begin{split}
&|\bar{b}(\omega,z_1;y) - \bar{b}(\omega,z_2;y)|^2+ |\bar{\sigma}(\omega,z_1;y)-
\bar{\sigma}(\omega,z_2;y)|^2\\
&+\int_{(0 < |x| < 1)}|\bar{F}(\omega,z_1,x;y) - \bar{F}(\omega,z_2,x;y)|^2 \, \nu(dx) \leq C(\K)\,
|z_1
- z_2|^2.
\end{split}
\end{equation}

\end{enumerate}
Then \eqref{reduced-fd-sde} has an $(\F_t)$ adapted strong solution $\{X_t\}$ with rcll paths. Pathwise uniqueness of solutions also holds, i.e. if $\{X^1_t\}$ is another such solution, then $P(X_t=X_t^1,t\geq0)=1$.
\end{theorem}

We now consider the SDE \eqref{fd-sde-sln}. The next result follows by the interlacing technique (see \cite[Example 1.3.13, pp. 50-51]{MR2512800}). The arguments run similar to \cite[Theorem 6.2.9]{MR2512800}.

\begin{theorem}[{\cite{fd-SDE}}]\label{interlacing-global-sde}
Suppose all the assumptions of Theorem \ref{nrm-bd-rndm-inl} hold. In addition, assume that \ref{G1} holds. Then there exists a unique rcll adapted solution to \eqref{fd-sde-sln}.
\end{theorem}

\subsection{Local Lipschitz coefficients}\label{sec:2-3}
Let $\widehat{\R^d}:=\R^d\cup\{\infty\}$ be the one point compactification of $\R^d$. The next result is an extension of Theorem \ref{interlacing-global-sde} for `global Lipschitz' coefficients to `local Lipschitz' coefficients.

\begin{theorem}[{\cite{fd-SDE}}]
\label{nrm-sqre-rndm-inl-fnl}
Let \ref{sigma-b}, \ref{F1}, \ref{F2}, \ref{F3}, \ref{loc-Lip} and \ref{G1} hold. Then there exists an $(\F_t)$ stopping time $\eta$ and an $(\F_t)$ adapted $\widehat{\R^d}$ valued process $\{X_t\}$ with rcll paths such that $\{X_t\}$ solves \eqref{fd-sde-sln} upto time $\eta$ and $X_t=\infty$ for $t\geq\eta$. Further $\eta$ can be identified as follows: $\eta=\lim_m\theta_m$ where $\{\theta_m\}$ are $(\F_t)$ stopping times defined by $\theta_m:=\inf\{t\geq0:|X_t|\geq m\}$. This is also pathwise unique in this sense: if $(\{X'_t\},\eta')$ is another such solution, then $P(X_t=X'_t, 0\leq t<\eta\wedge\eta')=1$.
\end{theorem}

Using Proposition \ref{ext-tau}, we now give explicit regularity assumptions on $\sigma, b, F$ which imply the `local Lipschitz' regularity \ref{loc-Lip} of $\bar \sigma, \bar b, \bar F$. The argument here is a variant of \cite[Proposition 3.8]{MR3687773}. For the sake of convenience, we state the result with deterministic $\sigma, b$ and $F$.

\begin{proposition}\label{f,bar-f}
Let $p>d+\frac{1}{2}$. Fix deterministic $b_i, \sigma_{ij} \in \Sc_{p+\tfrac{1}{2}}, 1 \leq i, j \leq d$ and $y \in \Sc_{-p}$. Assume \ref{F1} and \ref{F2}. Then for any positive integer $n$, there exists a constant $D_n>0$ such that for all $z_1,z_2\in \overline{\oo(0,n)}$ and $0<|x|<1$ \begin{equation}\label{f-bar1}
\begin{split}
&|\bar{b}(z_1;y) - \bar{b}(z_2;y)| \leq \|y\|_{-p} D_n\sup_i\|b_i\|_{p+\frac{1}{2}}\,
|z_1
- z_2|,\\
&|\bar{\sigma}(z_1;y)-
\bar{\sigma}(z_2;y)| \leq \|y\|_{-p} D_n\sup_{i,j}\|\sigma_{ij}\|_{p+\frac{1}{2}}\,
|z_1
- z_2|,\\
&|\bar{F}(z_1,x;y)-\bar{F}(z_2,x;y)|\leq C_x D_n \|y\|_{-p}|z_1-z_2|.
\end{split}
\end{equation}
In particular \ref{loc-Lip} follows, i.e. for any bounded set $\K$ in $\Sc_{-p}$ and any positive integer $n$, there exists a constant $D(\K,n)>0$ such that for all $z_1,z_2\in \overline{\oo(0,n)}$ and $y \in \K$,
\begin{equation}
\begin{split}
&|\bar{b}(z_1;y) - \bar{b}(z_2;y)| \leq D(\K,n)\,
|z_1
- z_2|,\\
&|\bar{\sigma}(z_1;y)-
\bar{\sigma}(z_2;y)| \leq D(\K,n)\,
|z_1
- z_2|,\\
&\int_{(0 < |x| < 1)}|\bar{F}(z_1,x;y) - \bar{F}(z_2,x;y)|^2 \, \nu(dx) \leq D(\K,n)\,
|z_1
- z_2|^2.
\end{split}
\end{equation}
\end{proposition}

\begin{proof}
We prove the result for $\bar F$. The results for $\bar \sigma, \bar b$ follow similarly.

If $z\in \R^d$ takes values in a bounded set, then using Proposition \ref{tau-x-estmte}, we conclude that corresponding $\tau_z y$ also takes values in some bounded set. Then we have for all $z_1,z_2\in \overline{\oo(0,n)}$,
\begin{align*}
|\bar{F}(z_1,x;y)-\bar{F}(z_2,x;y)|  = & |F(\tau_{z_1}y,x)-F(\tau_{z_2}y,x)|\\
\leq & C_x\|\tau_{z_1}y-\tau_{z_2}y\|_{-p-\frac{1}{2}}\ \ \text{[by \ref{F1}]}\\
=&C_x\sup_{\psi\in\Sc,\ \|\psi\|_{p+\frac{1}{2}}\leq1}|\langle\tau_{z_1}y-\tau_{z_2}y,\psi\rangle|\\
=&C_x\sup_{\psi\in\Sc,\ \|\psi\|_{p+\frac{1}{2}}\leq1}|\langle y,\tau_{-z_1}\psi-\tau_{-z_2}\psi\rangle|\\
\leq& C_x\sup_{\psi\in\Sc,\ \|\psi\|_{p+\frac{1}{2}}\leq1}\|y\|_{-p}\|\tau_{-z_1}\psi-\tau_{-z_2}\psi\|_{p}\\
\leq& C_x \|y\|_{-p} \sup_{\substack{\psi\in\Sc,\\ \|\psi\|_{p+\frac{1}{2}}\leq1}} D_n\|\psi\|_{p+\frac{1}{2}}|z_1-z_2|\ \ \text{[by Proposition \ref{ext-tau}]}\\
\leq&C_x D_n\|y\|_{-p}|z_1-z_2|.
\end{align*}
This proves the inequality for $\bar F$ in \eqref{f-bar1}. The other inequality for $\bar F$ follows from \ref{F2}.
\end{proof}

\section{Infinite dimensional SPDE}\label{sec:3}

We continue with the same notations and hypotheses as in Section~\ref{sec:2}. In this section, we study the existence and uniqueness of strong solutions to the following SPDE, viz.
\begin{equation}\label{Levy-SPDE}
\begin{split}
Y_t
&= \xi + \int_0^t
A(Y_{s-})\cdot dB_s + \int_0^t \widetilde L(Y_{s-})\, ds\\
&+ \int_0^t \int_{(0 < |x| < 1)} \left(\tau_{F(Y_{s-},x)}
-Id\right) \,
Y_{s-}\,\widetilde N(dsdx)\\
&+ \int_0^t \int_{(|x| \geq  1)} \left(\tau_{G(Y_{s-},x)}
-Id\right) \,Y_{s-}\,
N(dsdx),
\end{split}
\end{equation}
where $\xi$ is an $\Sc_{-p}$ valued $\F_0$-measurable random variable, $A = (A_1,\cdots, A_d)$ with $A_j:\Sc_{-p}\rightarrow\Sc_{-p-\frac{1}{2}} \subset \Sc_{-p-1}, j=1,2,\cdots,d$ and $\widetilde L:\Sc_{-p}\rightarrow\Sc_{-p-1}$ are defined as follows, for $\rho\in\Sc_{-p}$
\begin{equation}\label{op-A-L}
\begin{split}
&A_j\rho := -\sum_{i=1}^d  \langle\sigma,\rho\rangle_{ij}\, \partial_i\rho,\\
&\widetilde L(\rho) := L\rho + \int_{(0 < |x| < 1)} \left( \tau_{F(\rho, x)}
-Id +
\sum_{i=1}^d F^i(\rho,x)\, \partial_i\right) \rho\
\nu(dx),\\
&L\rho := \frac{1}{2}
\sum_{i,j=1}^d\big(\langle\sigma, \rho\rangle\langle\sigma, \rho\rangle^t\big)_{ij}\, \partial^2_{ij}\rho -\sum_{i=1}^d
\langle b,\rho\rangle_i\, \partial_i\rho.
\end{split}
\end{equation}

Given an $\Sc_{-p}$ valued adapted process $\{Y_t\}$ with rcll paths, the integrals (with respect to $B, \widetilde N$ and $\nu$) appearing in \eqref{Levy-SPDE} exist. For example, to show the existence of the integral with respect to $\widetilde N$, we need to establish $\Exp \sup_{t \geq 0} \int_0^{t\wedge\pi_n} \int_{(0 < |x| < 1)} \|\left(\tau_{F(Y_{s-},x)}
-Id\right) \,
Y_{s-}\|^2_{-p-1}  \nu(dx) ds < \infty$, for some increasing sequence of stopping times $\{\pi_n\}$ with $\pi_n \uparrow \infty$ a.s.. Now,
\begin{align*}
&\|\left(\tau_{F(Y_{s-},x)}
-Id\right) \,
Y_{s-}\|^2_{-p-1}\\
&= \sum_{m\in\mathbb{Z}^d_+} (2|m|+d)^{-2(p+1)} \left[f(1;F(Y_{s-},x),Y_{s-},m) - f(0;F(Y_{s-},x),Y_{s-},m) \right]^2\\
&= \sum_{m\in\mathbb{Z}^d_+} (2|m|+d)^{-2(p+1)} \left(\int_0^1 f^{\prime}(v;F(Y_{s-},x),Y_{s-},m)\,dv\right)^2,
\end{align*}
where the function $f$ is given by
\[f(v;z,\psi,m) := \inpr{\tau_{vz}\psi}{h_m},\, v\in [0,1],\]
for all fixed $z\in \R^d, \psi \in \Sc_{-p}, m \in \mathbb{Z}^d_+$. Using the fact that
\[f^\prime(v;z,\psi,m) = -\sum_{i=1}^d\inpr{z_i\partial_i\tau_{vz}\psi}{h_m}\]
and Lemma \ref{f-bd}, we can show
\begin{align*}
&\Exp \sup_{t \geq 0} \int_0^{t\wedge\pi_n} \int_{(0 < |x| < 1)} \|\left(\tau_{F(Y_{s-},x)}
-Id\right) \,
Y_{s-}\|^2_{-p-1}  \nu(dx) ds \\
&\leq C(n)\ \Exp \sup_{t \geq 0} \int_0^{t\wedge\pi_n} \int_{(0 < |x| < 1)} |F(Y_{s-},x)|^2  \nu(dx) ds < \infty,
\end{align*}
where, $\pi_n:=\inf\{t: \|Y_t\|_{-p} \geq n\}\wedge n$ and $\{C(n)\}$ denotes a sequence of positive real numbers. We omit the details here and provide the details (see \eqref{2nd-order-Taylor} and \eqref{term2} below) for a similar estimate involving a second order Taylor expansion.

Let $\delta$ be
an arbitrary state, viewed as an isolated point of $\hat\Sc_{-p} :=
\Sc_{-p} \cup \{\delta\}$. We make two definitions
extending \cite[Definition 3.1 and Definition 3.3]{MR3063763}.

\begin{definition}
Let $\xi$ be an $\Sc_{-p}$ valued $\F_0$-measurable random variable. By an $\hat\Sc_{-p}$ valued local strong solution of SPDE
\eqref{Levy-SPDE}, we
mean a pair $(\{Y_t\},\eta)$ where $\eta$ is an $(\F_t)$ stopping time and
$\{Y_t\}$ an $\hat\Sc_{-p}$ valued $(\F_t)$ adapted rcll process
such that
\begin{enumerate}
\item for all $\omega \in \Omega$, the map $Y_{\cdot}(\omega):[0,\eta(\omega))
\to \Sc_{-p}$ is well-defined and $Y_t(\omega) = \delta, \, t \geq
\eta(\omega)$.
\item a.s. the equality \eqref{Levy-SPDE} holds in $\Sc_{-p-1}$ for $0 \leq t <
\eta$.
\end{enumerate}

We say local strong solutions of SPDE \eqref{Levy-SPDE} are unique or
pathwise unique,
if given any
two $\hat\Sc_p(\R^d)$ valued strong solutions $(\{Y_t^1\},\eta^1)$ and
$(\{Y_t^2\},\eta^2)$, we have $P(Y_t^1 = Y_t^2,\, 0 \leq t <
\eta^1\wedge\eta^2)=1$.
\end{definition}

\subsection{Existence of solutions}\label{sec:3-1}
The next few results, viz. Theorem \ref{Levy-SPDE-existence}, Theorem \ref{uniq-via-monotonicity} and Lemma \ref{unqu} establish the existence and uniqueness of equation \eqref{Levy-SPDE} and also exhibit the translation invariance of the solutions.

\begin{theorem}\label{Levy-SPDE-existence}
Let \ref{sigma-b}, \ref{F1}, \ref{F2}, \ref{F3}, \ref{loc-Lip} and \ref{G1} hold. Consider the SDE \eqref{fd-sde-sln} with $\kappa = 0$ and
let $(\{U_t\},\eta)$ denote the unique local strong solution obtained by Theorem \ref{nrm-sqre-rndm-inl-fnl}. Then the $\Sc_{-p}$ valued process
 $\{Y_t\}$ defined by $Y_t := \tau_{U_t}\xi, t < \eta$
solves the SPDE \eqref{Levy-SPDE}. We set $Y_t := \delta, t \geq \eta$ so that $(\{Y_t\}, \eta)$ is a local strong solution of \eqref{Levy-SPDE}.
\end{theorem}

\begin{proof}
Note that $Y_{t-}=\tau_{U_{t-}}\xi, t < \eta$. By the It\^{o}
formula in Theorem \ref{random-initial}, a.s.
\begin{equation}\label{Ito-formula-random-initl-U}
\begin{split}
\tau_{U_t}\xi &=\xi - \sum_{i=1}^d \int_0^t
\partial_i\tau_{U_{s-}}\xi\,
dU^i_s + \frac{1}{2}\sum_{i,j=1}^d \int_0^t \partial_{ij}^2\tau_{U_{s-}}\xi\,
d[U^i,U^j]^c_s\\
&+\sum_{s \leq t}\left[\tau_{U_s}\xi - \tau_{U_{s-}}\xi +
\sum_{i=1}^d (\bigtriangleup U^i_s\,\partial_i\tau_{U_{s-}}\xi)\right], \, t < \eta.
\end{split}
\end{equation}

Observe that
\begin{equation}\label{exist-t4}
\bigtriangleup U^i_t =\bar F^i(U_{t-},\bigtriangleup
N_t;\xi)\indicator{(0 < |\bigtriangleup N_t| < 1)} +\bar G^i(U_{t-},\bigtriangleup
N_t;\xi)\indicator{(|\bigtriangleup N_t| \geq 1)}
\end{equation}
and hence
\begin{align*}
&\tau_{U_s}\xi - \tau_{U_{s-}}\xi
+\sum_{i=1}^d(\bigtriangleup
U^i_s\,\partial_i\tau_{U_{s-}}\xi)\\
&= \left( \tau_{\bigtriangleup U_s} -
Id\right)\tau_{U_{s-}}\xi
+\sum_{i=1}^d(\bigtriangleup
U^i_s\,\partial_i\tau_{U_{s-}}\xi)\\
&=\indicator{(0 < |\bigtriangleup N_s| < 1)} \left(\tau_{\bar F(U_{s-},\bigtriangleup
N_s;\xi)} -
Id +\sum_{i=1}^d\bar F^i(U_{s-},\bigtriangleup
N_s;\xi)\, \partial_i\right)\tau_{U_{s-}}\xi\\
&+ \indicator{(|\bigtriangleup N_s| \geq 1)}\left( \tau_{\bar G(U_{s-},\bigtriangleup
N_s;\xi)} -
Id +
\sum_{i=1}^d\bar G^i(U_{s-},\bigtriangleup
N_s;\xi)\, \partial_i\right)\tau_{U_{s-}}\xi.
\end{align*}
This observation yields a simplification of the fourth term of the right-hand side of \eqref{Ito-formula-random-initl-U}, viz.
\begin{align}\label{exist-t5}
\begin{split}
&\sum_{s \leq t}\left[\tau_{U_s}\xi - \tau_{U_{s-}}\xi +
\sum_{i=1}^d (\bigtriangleup U^i_s\,\partial_i\tau_{U_{s-}}\xi)\right]\\
&=\int_0^t\int_{(0<|x|<1)}\left(\tau_{\bar F(U_{s-},x;\xi)} -
Id +\sum_{i=1}^d\bar F^i(U_{s-},x;\xi)\, \partial_i\right)\tau_{U_{s-}}\xi\  N(dsdx)\\
&+\int_0^t\int_{(|x|>1)}\left( \tau_{\bar G(U_{s-},x;\xi)} -
Id +
\sum_{i=1}^d\bar G^i(U_{s-},x;\xi)\, \partial_i\right)\tau_{U_{s-}}\xi\ N(dsdx)\\
&=\int_0^t\int_{(0<|x|<1)}\left(\tau_{\bar F(U_{s-},x;\xi)} -
Id +\sum_{i=1}^d\bar F^i(U_{s-},x;\xi)\, \partial_i\right)\tau_{U_{s-}}\xi\ \widetilde N(dsdx)\\
&+\int_0^t\int_{(0<|x|<1)}\left(\tau_{\bar F(U_{s-},x;\xi)} -
Id +\sum_{i=1}^d\bar F^i(U_{s-},x;\xi)\, \partial_i\right)\tau_{U_{s-}}\xi\ \nu(dx)ds\\
&+\int_0^t\int_{(|x|>1)}\left( \tau_{\bar G(U_{s-},x;\xi)} -
Id \right)\tau_{U_{s-}}\xi\ N(dsdx)\\
&+\int_0^t\int_{(|x|>1)}\left(\sum_{i=1}^d\bar G^i(U_{s-},x;\xi)\, \partial_i\right)\tau_{U_{s-}}\xi\ N(dsdx).
\end{split}
\end{align}
Substituting \eqref{exist-t5} into \eqref{Ito-formula-random-initl-U} and simplifying the equality, we obtain \eqref{Levy-SPDE}. Here we have used the fact that the coefficients satisfy relations like $\bar F(U_{s-},x;\xi) = F(\tau_{U_{s-}}\xi,x) = F(Y_{s-},x)$ etc..
\end{proof}

\subsection{Uniqueness results via monotonicity arguments}\label{sec:3-2}
Existence of local strong solutions to SPDE \eqref{Levy-SPDE} follows from Theorem \ref{Levy-SPDE-existence}. Uniqueness of solutions to \eqref{Levy-SPDE} are the focus of this section. The proof of uniqueness is based on the `Monotonicity inequality', (see \cite[p. 29]{MR2479730}, \cite[p. 308]{MR1465436}, \cite[Section 3]{MR570795}]).

\begin{theorem}\label{uniq-via-monotonicity}
Let \ref{sigma-b}, \ref{F1}, \ref{F2}, \ref{F3}, \ref{loc-Lip} and \ref{G1} hold. In addition, assume that \ref{G2} holds. Then there exists a unique local strong solution to \eqref{Levy-SPDE}.
\end{theorem}

To establish the uniqueness, we first show that any local strong solution to equation \eqref{Levy-SPDE} is of a specific form. Even though this result is only used in the proof of Theorem \ref{uniq-via-monotonicity}, we state it separately in order to keep our arguments transparent.

\begin{lemma}\label{unqu}
Let the hypotheses of Theorem \ref{uniq-via-monotonicity} hold. Let $(\{Y_t\}, \eta)$ be an $\hat\Sc_{-p}$ valued local strong solution of \eqref{Levy-SPDE}. Define
\begin{equation}\label{Z}
\begin{split}
Z_t :&=  \int_0^t \inpr{b}{Y_{s-}}\, ds + \int_0^t \inpr{\sigma}{Y_{s-}}\cdot dB_s +\int_0^t \int_{(0 < |x| < 1)} F(Y_{s-},x)\, \widetilde
N(dsdx)\\
&+ \int_0^t \int_{(|x| \geq  1)} G(Y_{s-},x) \,
N(dsdx),
\end{split}
\end{equation}
for $0\leq t < \eta$. Then a.s. $Y_t = \tau_{Z_t} \xi, t < \eta$.
\end{lemma}
\begin{proof}
We follow the approach used in \cite[Lemma 3.6]{MR3063763}, but with initial condition $\xi$ random (see \cite[Theorem 3.14]{MR3687773}). Define $V_t := Y_t - \tau_{Z_t}\xi$. As done in Theorem \ref{Levy-SPDE-existence}, we simplify $\tau_{Z_t}\xi$ using the It\^o formula in Theorem \ref {random-initial}. Then using \eqref{Levy-SPDE}, we have a.s. for $0\leq t < \eta$ \begin{equation}\label{eqn-V}
\begin{split}
V_t &= \int_0^t
\bar A(s)V_{s-}\cdot dB_s + \int_0^t \bar L(s)V_{s-}\, ds\\
&+\int_0^t \int_{(0 < |x| < 1)} \left(
\tau_{F(Y_{s-},x)} -
Id + \sum_{i=1}^dF^i(Y_{s-},x)\, \partial_i\right) V_{s-}\,
\nu(dx)\,ds\\
&+\int_0^t \int_{(0 < |x| < 1)} \left(
\tau_{F(Y_{s-},x)} -
Id\right)V_{s-}\,\widetilde N(dsdx) + \int_0^t \int_{(|x| \geq  1)} \left(
\tau_{G(Y_{s-},x)} -
Id\right)V_{s-}\,
N(dsdx),
\end{split}
\end{equation}
where the bounded random linear operators $\bar A = (\bar A_1,\cdots, \bar A_d)$ with $\bar A_i:\Sc_{-p}\rightarrow \Sc_{-p-\frac{1}{2}}, i=1,\cdots, d$ and $\bar L:\Sc_{-p}\rightarrow \Sc_{-p-1}$ are defined as follows, for $\rho\in \Sc_{-p}$
\begin{align*}
\bar A_j(s,\omega) \rho &:= -\sum_{i=1}^d\langle\sigma(\omega),Y_{s-}(\omega)\rangle_{ij}\, \partial_i\rho,\  \ \ \ j=1,\cdots,d,\\
\bar L(s,\omega) \rho &:= \frac{1}{2}
\sum_{i,j=1}^d(\langle\sigma(\omega),Y_{s-}(\omega)\rangle\langle\sigma(\omega),Y_{s-}(\omega)\rangle^t)_{ij}\, \partial_{ij}^2 \rho -
\sum_{i=1}^d\langle b(\omega),Y_{s-}(\omega)\rangle_i\, \partial_i\rho.
\end{align*}
The following equation is obtained using It\^o formula for the norm $\|\cdot\|^2_{-p-1}$.
We have a.s. for $0\leq t < \eta$
\begin{equation}\label{V_t}
\begin{split}
&\|V_t\|^2_{-p-1} =  2\int_0^t \inpr[-p-1]{V_{s-}}{\bar A(s)V_{s-}}\cdot dB_s \\
&+ 2\int_0^t \inpr[-p-1]{V_{s-}}{\bar L(s)V_{s-}}\, ds +\int_0^t
\sum_{j=1}^d \|\bar A_j(s)V_{s-}\|^2_{-p-1}\, ds \\
&+\int_0^t \int_{(0 < |x| < 1)}  \left[\|\tau_{F(Y_{s-},x)}V_{s-}\|^2_{-p-1} - \|V_{s-}\|^2_{-p-1} + 2\sum_{i=1}^d\inpr[-p-1]{V_{s-}}{F^i(Y_{s-},x)\, \partial_i V_{s-}} \right]
\nu(dx)\,ds\\
&+\int_0^t \int_{(0 < |x| < 1)} \left[\|\tau_{F(Y_{s-},x)}V_{s-}\|^2_{-p-1} - \|V_{s-}\|^2_{-p-1} \right] \widetilde N(dsdx)\\
&+ \int_0^t \int_{(|x| \geq  1)} \left[\|\tau_{G(Y_{s-},x)}V_{s-}\|^2_{-p-1} - \|V_{s-}\|^2_{-p-1} \right]
N(dsdx).
\end{split}
\end{equation}

To establish the above equation we execute the following steps.
\begin{enumerate}[label=(\roman*)]
\item Recall that $h_m$ denotes the Hermite functions, where $m$ denotes multi-indices. Then from \eqref{eqn-V}, we obtain the equation satisfied by the real semimartingale $\{\inpr{V_t}{h_m}\}$.
\item Applying It\^{o} formula \cite[Theorem 4.4.7]{MR2512800} for the function $x \mapsto x^2$ we get an equation for the process $\{\inpr{V_t}{h_m}^2\}$.
\item We multiply the last equation by $(2|m|+d)^{-2(p+1)}$ and sum over $m$.
\end{enumerate}

Define $\pi_n:= \inf\{t: \max\{\|Y_t\|_{-p}, |Z_t|\} \geq n\}\wedge n \wedge \eta$. Note that $\|Y_{t-}^{\pi_n}\|_{-p-1} \leq \|Y_{t-}^{\pi_n}\|_{-p} \leq n$ and $|Z_{t-}^{\pi_n}| \leq n$. Hence the process $\{\|V_{t-}^{\pi_n}\|_{-p}\}$ is bounded and for $j=1,\cdots, d$, $\bar A_j(t):\Sc_{-p} \to \Sc_{-p-1}, t \leq \pi_n$ are bounded linear operators, bounded uniformly in $t$. This implies $\{\int_0^{t\wedge\pi_n} \inpr[-p-1]{V_{s-}}{\bar A(s)V_{s-}}\cdot dB_s\}$ is an $\mathcal L^2$ martingale.

Again, $\{\int_0^{t\wedge\pi_n} \int_{(0 < |x| < 1)} \left[\|\tau_{F(Y_{s-},x)}V_{s-}\|^2_{-p-1} - \|V_{s-}\|^2_{-p-1} \right] \widetilde N(dsdx)\}$ is an $\mathcal L^2$ martingale, since $\Exp \sup_{t \geq 0} \int_0^{t\wedge\pi_n} \int_{(0 < |x| < 1)} \left[\|\tau_{F(Y_{s-},x)}V_{s-}\|^2_{-p-1} - \|V_{s-}\|^2_{-p-1} \right]^2 \nu(dx) ds < \infty$. This integrability condition follows using a first order Taylor expansion of the function $f$ defined in \eqref{Taylor-required-f}. We omit the details here and provide the details (see \eqref{2nd-order-Taylor} and \eqref{term2} below) for the corresponding second order Taylor expansion used in estimating Term 2 of \eqref{Exp-norm-Ito}. However, \eqref{term2} requires \eqref{spl-mono}, while the proof of the integrability condition above involving the first order Taylor expansion uses \eqref{1st-order-monotonicity} instead. Since a.s. $V$ has (at most) countably many jumps, taking expectation on both sides of \eqref{V_t} we get
\begin{equation}\label{Exp-norm-Ito}
\begin{split}
&\Exp\|V_{t\wedge\pi_n}\|^2_{-p-1} =  \Exp\int_0^{t\wedge\pi_n} \left[ 2\inpr[-p-1]{V_{s}}{\bar L(s)V_{s}} +
\|\bar A(s)V_{s}\|^2_{HS(-p-1)}\right] ds \\
&+ \Exp\int_0^{t\wedge\pi_n} \int_{(0 < |x| < 1)}  \left[\|\tau_{F(Y_{s-},x)}V_{s}\|^2_{-p-1} - \|V_{s}\|^2_{-p-1} + 2\sum_{i=1}^d\inpr[-p-1]{V_{s}}{F^i(Y_{s-},x)\, \partial_i V_{s}} \right]
\nu(dx)\,ds\\
&+ \Exp \int_0^{t\wedge\pi_n} \int_{(|x| \geq  1)} \left[\|\tau_{G(Y_{s-},x)}V_{s}\|^2_{-p-1} - \|V_{s}\|^2_{-p-1} \right]
\nu(dx)\, ds\\
&= \text{Term 1} + \text{Term 2} + \text{Term 3}.
\end{split}
\end{equation}
We now prove certain estimates of these terms. Some positive constants appearing in these calculations may be written by $C$ and may change their values from line to line, but will depend on $n$ and $d$.

\underline{Estimate for Term 1:} By assumption \ref{sigma-b}, the coefficients in $\bar L(s), \bar A(s)$ are bounded for $s \leq \pi_n$. Hence applying the  Monotonicity inequality (Theorem \ref{constant-monotonicity}), we get
\begin{equation}\label{term1}
\Exp \int_0^{t\wedge\pi_n} \left[ 2\inpr[-p-1]{V_{s}}{\bar L(s)V_{s}} +
\|\bar A(s)V_{s}\|^2_{HS(-p-1)}\right] ds \leq C \Exp \int_0^{t\wedge\pi_n}  \|V_s\|_{-p-1}^2 \,ds,
\end{equation}
where $C$ is some positive constant.

\underline{Estimate for Term 2:} First we need a special case of the Monotonicity inequality (Theorem \ref{constant-monotonicity}), viz. we need an explicit form of the constant in this special case, to ensure certain integrability conditions. To prove this, we use an alternative proof of the Monotonicity inequality already given in \cite{MR3331916}.

By \cite[Theorem 2.5]{MR3331916}, for each $1 \leq i \leq d$, there exists
a bounded operator
$\mathbb{T}_i:\Sc_{-p-1}\to\Sc_{-p-1}$ such that the adjoint operator $\partial_i^\ast$ has the form $\partial_i^\ast = -\partial_i + \mathbb{T}_i$ on $\Sc$. Recall that $\partial_i:\Sc_{-p-\tfrac{1}{2}}\to\Sc_{-p-1}, 1 \leq i \leq d$ are bounded linear operators. Hence it is easy to see that $\partial_i^\ast = -\partial_i + \mathbb{T}_i$ on $\Sc_{-p-\tfrac{1}{2}}$. Moreover, by \cite[Lemma 2.6]{MR3331916}, the map $\inpr[-p-1]{\partial_i (\cdot)}{\mathbb{T}_j(\cdot)}:\Sc \times \Sc \rightarrow \R$ defined by
\[(\phi,\psi)\mapsto \inpr[-p-1]{\partial_i \phi}{\mathbb{T}_j\psi}, \; \forall \phi,\psi \in \Sc\]
extends to a bounded bilinear form on $\Sc_{-p-1}\times \Sc_{-p-1}$. Let $\alpha = (\alpha_1,\cdots,\alpha_d)^t \in \R^d$ and $\phi \in \Sc_{-p}$ be chosen arbitrarily. Then there exists a positive constant $R$, not depending on $\alpha$ and $\phi$, such that
\begin{equation}\label{spl-mono}
\begin{split}
\sum_{i,j=1}^d \alpha_i \alpha_j  \left[\inpr[-p-1]{ \partial_i\phi}{ \partial_j\phi} + \inpr[-p-1]{\phi}{\partial^2_{ij}\phi}\right]&=\sum_{i,j=1}^d \alpha_i \alpha_j \inpr[-p-1]{\mathbb{T}_i \phi}{\partial_j \phi}\\
&\leq R \|\phi\|_{-p-1}^2 \left(\sum_{i=1}^d |\alpha_i|\right)^2\\
&\leq dR \|\phi\|_{-p-1}^2 |\alpha|^2.
\end{split}
\end{equation}

To estimate Term 2, we use a second order Taylor expansion described below. Observe that
\begin{equation}\label{2nd-order-Taylor}
\begin{split}
&\|\tau_{F(Y_{s-},x)}V_{s}\|^2_{-p-1} - \|V_{s}\|^2_{-p-1} + 2\sum_{i=1}^d\inpr[-p-1]{V_{s}}{F^i(Y_{s-},x)\, \partial_i V_{s}} \\
&= \sum_{m\in\mathbb{Z}^d_+} (2|m|+d)^{-2(p+1)} \left[\inpr{\tau_{F(Y_{s-},x)}V_{s}}{h_m}^2 - \inpr{V_{s}}{h_m}^2 + 2 \sum_{i=1}^d \inpr{V_{s}}{h_m}\inpr{F^i(Y_{s-},x)\, \partial_i V_{s}}{h_m} \right] \\
&= \sum_{m\in\mathbb{Z}^d_+} (2|m|+d)^{-2(p+1)} \left[f(1;F(Y_{s-},x),V_s,m) - f(0;F(Y_{s-},x),V_s,m) - f^\prime(0;F(Y_{s-},x),V_s,m)\right]\\
&= \sum_{m\in\mathbb{Z}^d_+} (2|m|+d)^{-2(p+1)} \int_0^1 \int_0^r f^{\prime\prime}(v;F(Y_{s-},x),V_s,m)\,dv\, dr,
\end{split}
\end{equation}
where the function $f$ is given by
\begin{equation}\label{Taylor-required-f}
f(v;z,\psi,m) := \inpr{\tau_{vz}\psi}{h_m}^2,\, v\in [0,1],
\end{equation}
for all fixed $z\in \R^d, \psi \in \Sc_{-p}, m \in \mathbb{Z}^d_+$. Note that
\[f^\prime(v;z,\psi,m) = -2\sum_{i=1}^d\inpr{\tau_{vz}\psi}{h_m}\inpr{z_i\partial_i\tau_{vz}\psi}{h_m},\]
and
\[f^{\prime\prime}(v;z,\psi,m) = 2\sum_{i,j} z_i z_j \left[\inpr{\partial_i\tau_{vz}\psi}{h_m} \inpr{\partial_j\tau_{vz}\psi}{h_m} + \inpr{\tau_{vz}\psi}{h_m}\inpr{\partial_{ij}^2\tau_{vz}\psi}{h_m}\right].\]
Now, using \eqref{spl-mono} and Proposition \ref{tau-x-estmte}$(a)$, we have
\begin{align*}
&\|\tau_{F(Y_{s-},x)}V_{s}\|^2_{-p-1} - \|V_{s}\|^2_{-p-1} + 2\sum_{i=1}^d\inpr[-p-1]{V_{s}}{F^i(Y_{s-},x)\, \partial_i V_{s}} \\
&= 2 \int_0^1 \int_0^r
\sum_{i,j} F^i(Y_{s-},x) F^j(Y_{s-},x)\\ &\qquad\qquad \times \left[\inpr[-p-1]{\partial_i\tau_{vF(Y_{s-},x)}V_s}{\partial_j \tau_{vF(Y_{s-},x)}V_s}
+\inpr[-p-1]{\tau_{vF(Y_{s-},x)}V_s}{\partial^2_{ij} \tau_{vF(Y_{s-},x)}V_s} \right]dv\, dr\\
&\leq C |F(Y_{s-},x)|^2 \int_0^1 \int_0^r \|\tau_{vF(Y_{s-},x)}V_s\|_{-p-1}^2 \, dv\, dr\\
&\leq C |F(Y_{s-},x)|^2 \|V_s\|_{-p-1}^2 \int_0^1 \int_0^r (P(|vF(Y_{s-},x)|))^2 \, dv\, dr,
\end{align*}
where $P$ is some real polynomial of degree $2(\lfloor p+1 \rfloor +1)$. Now $\|Y_{t-}^{\pi_n}\|_{-p} \leq n$ and hence, by Lemma \ref{f-bd}(i), $P(|vF(Y_{s-}^{\pi_n},x)|))$ is also bounded for all $s \leq \pi_n$ and $v \in [0,1]$. Then
\begin{equation}\label{term2}
\begin{split}
&\Exp\int_0^{t\wedge\pi_n} \int_{(0 < |x| < 1)} \left[\|\tau_{F(Y_{s-},x)}V_{s}\|^2_{-p-1} - \|V_{s}\|^2_{-p-1} + 2\sum_{i=1}^d\inpr[-p-1]{V_{s}}{F^i(Y_{s-},x)\, \partial_i V_{s}} \right]
\nu(dx)\,ds\\
&\leq C \Exp \int_0^{t\wedge\pi_n} \int_{(0 < |x| < 1)}  |F(Y_{s-},x)|^2 \|V_s\|_{-p-1}^2 \int_0^1 \int_0^r (P(|vF(Y_{s-},x)|))^2 \, dv\, dr\,
\nu(dx)\,ds\\
&\leq C \Exp \int_0^{t\wedge\pi_n} \int_{(0 < |x| < 1)}  |F(Y_{s-},x)|^2 \|V_s\|_{-p-1}^2 \,
\nu(dx)\,ds\\
&\leq C \Exp \int_0^{t\wedge\pi_n}  \|V_s\|_{-p-1}^2 \,ds.
\end{split}
\end{equation}
In the last step above, we have used Lemma \ref{f-bd}(ii).

\underline{Estimate for Term 3:} Since $\{Y_{t-}^{\pi_n}\}$ is bounded in $\Sc_{-p}$, using \ref{G2} we get a bound for $G(Y_{s-}, x)$ when $x \in \overline{\oo(0,1)}, s \leq \pi_n$. Applying Proposition \ref{tau-x-estmte}, we have
\begin{equation}\label{term3}
\Exp \int_0^{t\wedge\pi_n} \int_{(|x| \geq  1)} \left[\|\tau_{G(Y_{s-},x)}V_{s}\|^2_{-p-1} - \|V_{s}\|^2_{-p-1} \right]
\nu(dx)\, ds
\leq C \Exp \int_0^{t\wedge\pi_n}  \|V_s\|_{-p-1}^2 \,ds,
\end{equation}
where $C$ is some positive constant.

From \eqref{Exp-norm-Ito}, \eqref{term1}, \eqref{term2} and \eqref{term3}, we get
\begin{equation}\label{gronwl-V}
\Exp\|V_t^{\pi_n}\|^2_{-p-1}\leq C \Exp \int_0^{t\wedge\pi_n}  \|V_s\|_{-p-1}^2 \,ds \leq C \Exp \int_0^t  \|V_s^{\pi_n}\|_{-p-1}^2 \,ds.
\end{equation}

By Gronwall's inequality, we have a.s. $V_t^{\pi_n} = 0, \forall t$. Hence, a.s. $V_t = 0, t < \eta$ and $Y_t = \tau_{Z_t}\xi, t < \eta$, where $\{Z_t\}$ is given by \eqref{Z}.
\end{proof}

\begin{proof}[Proof of Theorem \ref{uniq-via-monotonicity}]
Existence of local strong solutions to SPDE \eqref{Levy-SPDE} follows from Theorem \ref{Levy-SPDE-existence}. The uniqueness argument follows as in \cite{MR3063763}. Given any local strong solution $(\{Y_t\}, \eta)$, by Lemma \ref{unqu}, we have $Y_t = \tau_{Z_t}\xi$ where $\{Z_t\}$ is given by \eqref{Z}. Hence, \eqref{Z} becomes
\begin{align*}
Z_t &=  \int_0^t \bar b(Z_{s-};\xi)\, ds + \int_0^t \bar\sigma(Z_{s-};\xi)\cdot dB_s + \int_0^t \int_{(0 < |x| < 1)} \bar F(Z_{s-},x;\xi)\, \widetilde
N(dsdx)\\
&+ \int_0^t \int_{(|x| \geq  1)} \bar G(Z_{s-},x;\xi) \,
N(dsdx).
\end{align*}
By Theorem \ref{nrm-sqre-rndm-inl-fnl}, $(\{Z_t\}, \eta)$ is pathwise unique. Since $Y_t = \tau_{Z_t}\xi$, $(\{Y_t\}, \eta)$ is also pathwise unique.
\end{proof}

\begin{remark}\label{term2-via-mon}
The estimate of Term 2 in Lemma \ref{unqu} follows easily in simple situations. For example, let $F$ be bounded and the measure $\nu$ be finite. Using the boundedness of the translation operator (Proposition \ref{tau-x-estmte}), the term $\Exp\int_0^{t\wedge\pi_n} \int_{(0 < |x| < 1)}  \left[\|\tau_{F(Y_{s-},x)}V_{s}\|^2_{-p-1} - \|V_{s}\|^2_{-p-1} \right]
\nu(dx) ds$ can be estimated. To handle the remaining part, observe that $\inpr[-p-1]{\phi}{\partial_i \phi} = - \inpr[-p-1]{\partial_i \phi}{\phi} + \inpr[-p-1]{\mathbb{T}_i \phi}{\phi}$ for all $\phi \in \Sc_{-p}$ with $\mathbb{T}_i$ as described in the proof of Lemma \ref{unqu}. Then
\begin{equation}\label{1st-order-monotonicity}
2 \inpr[-p-1]{\phi}{\partial_i \phi} = \inpr[-p-1]{\mathbb{T}_i \phi}{\phi} \forall \phi \in \Sc_{-p},
\end{equation}
and the relevant term can be dominated by a constant multiple of $\|\phi\|_{-p-1}^2$, which in our case gives a bound involving $\|V_s\|_{-p-1}^2$.
\end{remark}

\begin{example}
Examples of the SPDEs we have considered are as follows.
\begin{enumerate}
\item Consider $F\equiv G \equiv 0$ in \eqref{Levy-SPDE}, with deterministic initial condition $\xi = \phi\in\Sc_{-p}$. Let $\sigma, b\in\Sc_p$ be deterministic such that $\langle\sigma,\tau_z\phi\rangle$ and $\langle b,\tau_z\phi\rangle$ are locally Lipschitz in $z$. Define $\bar\sigma(\omega,z;y):=\langle\sigma,\tau_z\phi\rangle$ and $\bar b(\omega,z;y):=\langle b,\tau_z\phi\rangle$ for $y\in\Sc_{-p}, \omega \in \Omega$. Then \ref{loc-Lip} holds. Applying Theorem \ref{nrm-sqre-rndm-inl-fnl}, existence and uniqueness of finite dimensional SDEs follow  for $\kappa = 0$. By Theorem  \ref{Levy-SPDE-existence} and Theorem \ref{uniq-via-monotonicity}, existence and uniqueness of the corresponding SPDE are established. Our results therefore implies Theorem 3.4 of \cite{MR3063763}.
\item Continue with $F\equiv G \equiv 0$ and $\kappa = 0$ as in the previous example, but consider $\xi$ with $\Exp \|\xi\|_{-p}^2 < \infty$. Then our results imply the results on existence and uniqueness of solutions studied in \cite[Section 3]{MR3687773}.
\item We consider $\sigma \equiv b \equiv G\equiv 0$. Let $F:\Sc_{-p}\times\oo(0,1)\to\R^d$, defined by $F(y,x):=x$. Then we have the existence and uniqueness of the following SPDE
\[
Y_t=\xi+\int_0^t\int_{(0<|x|<1)}\left(\tau_x-Id+\sum_{i=1}^dx_i\partial_i\right)Y_{s-}\nu(dx)ds
+\int_0^t\int_{(0<|x|<1)}\left(\tau_x-Id\right)Y_{s-}\,\widetilde N(dsdx).\]
\end{enumerate}

\end{example}

\subsection{Uniqueness via interlacing}\label{sec:3-3}
Using results of the previous subsection, the existence and uniqueness of local strong solutions of the reduced equation corresponding to \eqref{Levy-SPDE} follows, which is the case when $G \equiv 0$. In this subsection, we use the result for the reduced equation and use an interlacing argument to attach large jumps to obtain existence and uniqueness of local strong solutions of SPDE \eqref{Levy-SPDE}. This approach allows us to drop the assumption \ref{G2} which was used in Theorem \ref{uniq-via-monotonicity}.

\begin{theorem}\label{complt-eqn}
Let \ref{sigma-b}, \ref{F1}, \ref{F2}, \ref{F3}, \ref{loc-Lip} and \ref{G1} hold. Then we have the existence and uniqueness of the local strong solutions of \eqref{Levy-SPDE}.
\end{theorem}
\begin{proof}
Existence of local strong solutions to SPDE \eqref{Levy-SPDE} follows from Theorem \ref{Levy-SPDE-existence}. We use the interlacing procedure described in \cite[Example 1.3.13, pp. 50-51]{MR2512800} to establish the uniqueness.

Let $\{\pi_n\}_{n\in\N}$ be the arrival times for the jumps of the compound Poisson process $\{P_t\}_{t\geq0}$, where each $P_t=\int_{(|x|\geq1)}xN(t,dx)$. Let $(\{Y_t\}, \eta)$ be a local strong solution of \eqref{Levy-SPDE}. Since, a.s. $\pi_n\uparrow\infty$, the stochastic interval $[0,\eta)$ can be decomposed as a disjoint union $\bigcup_{n=0}^\infty [\pi_n\wedge\eta, \pi_{n+1}\wedge\eta)$, where $\pi_0 = 0$. We now construct an $\R^d$ valued adapted rcll process $\{Z_t\}$ such that the following equalities hold; a.s.
\begin{equation}\label{ZviaY}
\begin{split}
Z_t &=  \int_0^t \inpr{b}{Y_{s-}}\, ds + \int_0^t \inpr{\sigma}{Y_{s-}}\cdot dB_s +\int_0^t \int_{(0 < |x| < 1)} F(Y_{s-},x)\, \widetilde
N(dsdx)\\
&+ \int_0^t \int_{(|x| \geq  1)} G(Y_{s-},x) \,
N(dsdx), \ t < \eta,
\end{split}
\end{equation}

\begin{equation}\label{YviaZ}
Y_t = \tau_{Z_t} \xi, \ t < \eta,
\end{equation}

\begin{equation}\label{ZviaSDE}
\begin{split}
Z_t &=  \int_0^t \bar b(Z_{s-};\xi)\, ds + \int_0^t \bar\sigma(Z_{s-};\xi)\cdot dB_s + \int_0^t \int_{(0 < |x| < 1)} \bar F(Z_{s-},x;\xi)\, \widetilde
N(dsdx)\\
&+ \int_0^t \int_{(|x| \geq  1)} \bar G(Z_{s-},x;\xi) \,
N(dsdx), \ t < \eta.
\end{split}
\end{equation}
As pointed out in the proof of Theorem \ref{uniq-via-monotonicity}, \eqref{ZviaSDE} follows from \eqref{ZviaY} and \eqref{YviaZ}. We verify the claimed equalities on successive time intervals.

Comparing $\pi_n$'s and $\eta$, two cases arise viz. either $\pi_n \leq \eta < \pi_{n+1}$ for some $n \geq 0$ or $\pi_n < \eta, \forall n$. We consider the first case. The proof for the second case is similar.

If $n = 0$, i.e. $\pi_0 \leq \eta < \pi_1$, define $\{Z_t\}$ by the right hand side of \eqref{ZviaY}. Since there is no large jump, the equality in \eqref{YviaZ} follows as in Lemma \ref{unqu} and \eqref{ZviaSDE} also follows.

Now assume $n \geq 1$, i.e. $\pi_0 < \pi_1 < \cdots < \pi_n \leq \eta < \pi_{n+1}$.

On $[0,\pi_1)$, define $\{Z_t\}$ by the right hand side of \eqref{ZviaY}. Since there is no large jump, the equality in \eqref{YviaZ} follows as in Lemma \ref{unqu} and \eqref{ZviaSDE} also follows.

At $t = \pi_1$, define $Z_t := Z_{t-} + \bar G( Z_{t-},\triangle P_t;\xi)$. Then the equality in \eqref{ZviaSDE} holds true. Note that $Y_{t-} = \tau_{Z_{t-}} \xi$ on $(0, \pi_1]$. By It\^{o} formula in Theorem \ref{random-initial} the equality in \eqref{YviaZ} follows. Consequently, equality in \eqref{ZviaY} follows.

On $(\pi_1, \pi_2)$, define $\{Z_t\}$ by the right hand side of \eqref{ZviaY}.  Observe that there is no contribution of the large jump at $\pi_1$ in the difference $Y_t - \tau_{Z_t}\xi$. Hence, arguing as in Lemma \ref{unqu}, the equality in \eqref{YviaZ} follows for the time interval $(\pi_1, \pi_2)$. Equality in \eqref{ZviaSDE} also follows for the same time interval.

At $t = \pi_2$, define $Z_t := Z_{t-} + \bar G( Z_{t-},\triangle P_t;\xi)$. We verify the equalities at $t = \pi_2$ as in the case $t = \pi_1$.

Continuing this way, we construct $\{Z_t\}$. Since \eqref{ZviaSDE} holds, the uniqueness of $\{Z_t\}$ follows from Theorem \ref{nrm-sqre-rndm-inl-fnl}. Since $Y_t = \tau_{Z_t}\xi$, $\{Y_t\}$ is also unique.

This completes the proof.
\end{proof}

\textbf{Acknowledgement:} The first author would like to acknowledge the fact that he was supported by the NBHM (National Board for Higher Mathematics, under Department of Atomic Energy, Government of India) Post Doctoral Fellowship. The third author would like to acknowledge the fact that he was partially supported by the ISF-UGC research grant.

\bibliographystyle{plain}
%\bibliography{references}

\begin{thebibliography}{10}

\bibitem{MR2512800}
David Applebaum.
\newblock {\em L\'evy processes and stochastic calculus}, volume 116 of {\em
  Cambridge Studies in Advanced Mathematics}.
\newblock Cambridge University Press, Cambridge, second edition, 2009.

\bibitem{JOTP-erratum}
Suprio Bhar.
\newblock Correction to: An {I}t\=o {F}ormula in the {S}pace of {T}empered
  {D}istributions.
\newblock {\em J. Theoret. Probab.}, 30(4):1786--1787, 2017.

\bibitem{MR3647067}
Suprio Bhar.
\newblock An {I}t\=o {F}ormula in the {S}pace of {T}empered {D}istributions.
\newblock {\em J. Theoret. Probab.}, 30(2):510--528, 2017.

\bibitem{MR3687773}
Suprio Bhar.
\newblock Stationary solutions of stochastic partial differential equations in
  the space of tempered distributions.
\newblock {\em Commun. Stoch. Anal.}, 11(2):169--193, 2017.

\bibitem{mild-soln}
Suprio Bhar, Rajeev Bhaskaran, and Barun Sarkar.
\newblock {Solutions of SPDE's associated with a stochastic flow}.
\newblock {\em arXiv:1706.06262 [math.PR] (Preprint)}.

\bibitem{MR3331916}
Suprio Bhar and B.~Rajeev.
\newblock Differential operators on {H}ermite {S}obolev spaces.
\newblock {\em Proc. Indian Acad. Sci. Math. Sci.}, 125(1):113--125, 2015.

\bibitem{fd-SDE}
Suprio Bhar and Barun Sarkar.
\newblock {Parametric family of SDEs driven by L\'evy noise}.
\newblock {\em arXiv:1801.06773 [math.PR] (Preprint)}.

\bibitem{MR1849252}
Tomas Bj\"ork and Bent~Jesper Christensen.
\newblock Interest rate dynamics and consistent forward rate curves.
\newblock {\em Math. Finance}, 9(4):323--348, 1999.

\bibitem{MR1822777}
Tomas Bj\"ork and Lars Svensson.
\newblock On the existence of finite-dimensional realizations for nonlinear
  forward rate models.
\newblock {\em Math. Finance}, 11(2):205--243, 2001.

\bibitem{MR1242575}
Donald~A. Dawson.
\newblock Measure-valued {M}arkov processes.
\newblock In {\em \'{E}cole d'\'{E}t\'e de {P}robabilit\'es de {S}aint-{F}lour
  {XXI}---1991}, volume 1541 of {\em Lecture Notes in Math.}, pages 1--260.
  Springer, Berlin, 1993.

\bibitem{MR1280712}
Eugene~B. Dynkin.
\newblock {\em An introduction to branching measure-valued processes}, volume~6
  of {\em CRM Monograph Series}.
\newblock American Mathematical Society, Providence, RI, 1994.

\bibitem{MR3227060}
Damir Filipovi\'c, Stefan Tappe, and Josef Teichmann.
\newblock Invariant manifolds with boundary for jump-diffusions.
\newblock {\em Electron. J. Probab.}, 19:no. 51, 28, 2014.

\bibitem{MR2479730}
L.~Gawarecki, V.~Mandrekar, and B.~Rajeev.
\newblock {Linear stochastic differential equations in the dual of a
  multi-{H}ilbertian space}.
\newblock {\em Theory Stoch. Process.}, 14(2):28--34, 2008.

\bibitem{MR2590157}
L.~Gawarecki, V.~Mandrekar, and B.~Rajeev.
\newblock {The monotonicity inequality for linear stochastic partial
  differential equations}.
\newblock {\em Infin. Dimens. Anal. Quantum Probab. Relat. Top.},
  12(4):575--591, 2009.

\bibitem{MR2560625}
Leszek Gawarecki and Vidyadhar Mandrekar.
\newblock {\em Stochastic differential equations in infinite dimensions with
  applications to stochastic partial differential equations}.
\newblock Probability and its Applications (New York). Springer, Heidelberg,
  2011.

\bibitem{MR562914}
Takeyuki Hida.
\newblock {\em {Brownian motion}}, volume~11 of {\em {Applications of
  Mathematics}}.
\newblock Springer-Verlag, New York, 1980.
\newblock Translated from the Japanese by the author and T. P. Speed.

\bibitem{MR771478}
Kiyosi It{\=o}.
\newblock {\em Foundations of stochastic differential equations in
  infinite-dimensional spaces}, volume~47 of {\em CBMS-NSF Regional Conference
  Series in Applied Mathematics}.
\newblock Society for Industrial and Applied Mathematics (SIAM), Philadelphia,
  PA, 1984.

\bibitem{MR1465436}
Gopinath Kallianpur and Jie Xiong.
\newblock {\em {Stochastic differential equations in infinite-dimensional
  spaces}}.
\newblock {Institute of Mathematical Statistics Lecture Notes---Monograph
  Series, 26}. Institute of Mathematical Statistics, Hayward, CA, 1995.
\newblock Expanded version of the lectures delivered as part of the 1993
  Barrett Lectures at the University of Tennessee, Knoxville, TN, March 25--27,
  1993, With a foreword by Balram S. Rajput and Jan Rosinski.

\bibitem{MR570795}
N.~V. Krylov and B.~L. Rozovski{\u\i}.
\newblock Stochastic evolution equations.
\newblock In {\em Current problems in mathematics, {V}ol. 14 ({R}ussian)},
  pages 71--147, 256. Akad. Nauk SSSR, Vsesoyuz. Inst. Nauchn. i Tekhn.
  Informatsii, Moscow, 1979.

\bibitem{MR1837298}
B.~Rajeev.
\newblock {From {T}anaka's formula to {I}to's formula: distributions, tensor
  products and local times}.
\newblock In {\em {S{\'e}minaire de {P}robabilit{\'e}s, {XXXV}}}, volume 1755
  of {\em {Lecture Notes in Math.}}, pages 371--389. Springer, Berlin, 2001.

\bibitem{MR3063763}
B.~Rajeev.
\newblock Translation invariant diffusion in the space of tempered
  distributions.
\newblock {\em Indian J. Pure Appl. Math.}, 44(2):231--258, 2013.

\bibitem{MR3517629}
B.~Rajeev and K.~Suresh~Kumar.
\newblock A class of stochastic differential equations with pathwise unique
  solutions.
\newblock {\em Indian J. Pure Appl. Math.}, 47(2):343--355, 2016.

\bibitem{MR1999259}
B.~Rajeev and S.~Thangavelu.
\newblock {Probabilistic representations of solutions to the heat equation}.
\newblock {\em Proc. Indian Acad. Sci. Math. Sci.}, 113(3):321--332, 2003.

\bibitem{MR2373102}
B.~Rajeev and S.~Thangavelu.
\newblock {Probabilistic representations of solutions of the forward
  equations}.
\newblock {\em Potential Anal.}, 28(2):139--162, 2008.

\bibitem{MR664333}
A.~S. {\"U}st{\"u}nel.
\newblock {A generalization of {I}t{\^o}'s formula}.
\newblock {\em J. Funct. Anal.}, 47(2):143--152, 1982.

\bibitem{MR3075420}
S.~R.~S. Varadhan.
\newblock {\em Collected papers. {IV}. {P}article systems and their large
  deviations}.
\newblock Hindustan Book Agency, New Delhi; Springer, Heidelberg, 2012.
\newblock Edited by Rajendra Bhatia, Abhay Bhatt and K. R. Parthasarathy.

\bibitem{MR3235846}
Jie Xiong.
\newblock {\em Three classes of nonlinear stochastic partial differential
  equations}.
\newblock World Scientific Publishing Co. Pte. Ltd., Hackensack, NJ, 2013.

\end{thebibliography}

\end{document}